\documentclass[11pt,psamsfonts,twoside]{article}
\usepackage{amssymb}
\usepackage{amssymb}
\usepackage{amssymb}
\usepackage{amssymb}
\usepackage{mathrsfs}
\usepackage{amsmath}
\usepackage{cases}
\usepackage{amsmath}
\usepackage{amsfonts}
\usepackage{amsfonts,latexsym,bm}
\usepackage{amsmath,amsthm}
\usepackage{amssymb,amscd}
\usepackage{amsfonts,amsbsy}
\usepackage{fancyhdr,graphicx}
\usepackage[dvips]{psfrag}
\usepackage{indentfirst}
\numberwithin{equation}{section}

\newtheorem {definition} {Definition} [section]
\newtheorem {theorem} [definition]{Theorem}
\newtheorem {proposition} {Proposition}[section]

\newtheorem {lemma}  [proposition]{Lemma}

 \allowdisplaybreaks

\newcommand{\supp}{{\rm supp}}
\newcommand{\diva}{{\rm div}}

\newcommand{\rn}{\mathbb{R}^n}
\newcommand{\half}{\frac{1}{2}}

\begin{document}
\setlength{\parindent}{4ex} \setlength{\parskip}{1ex}
\setlength{\oddsidemargin}{0mm} \setlength{\evensidemargin}{0mm}
\title{{Well-posedness of the Viscous Boussinesq System in Besov Spaces of Negative Order Near Index
$s=-1$ }}

\author{Chao Deng, Shangbin Cui\footnote{E-mail:
deng315@yahoo.com.cn, cuisb3@yahoo.com.cn}\\ [0.2cm] {\small Department of Mathematics,
Sun Yat-Sen
University, Guangzhou, }\\
{\small Guangdong 510275, P. R. of China }}
\date{}

\maketitle
\begin{abstract}
This paper is concerned with well-posedness of the Boussinesq
system. We prove that the $n$ ($n\ge2$) dimensional Boussinesq
system is well-psoed for small initial data $(\vec{u}_0,\theta_0)$
($\nabla\cdot\vec{u}_0=0$) either in
$({B}^{-1}_{\infty,1}\cap{B^{-1,1}_{\infty,\infty}})\times{B}^{-1}_{p,r}$
or in
${B^{-1,1}_{\infty,\infty}}\times{B}^{-1,\varepsilon}_{p,\infty}$ if
$r\in[1,\infty]$, $\varepsilon>0$ and $p\in(\frac{n}{2},\infty)$,
where $B^{s,\varepsilon}_{p,q}$ ($s\in\mathbb{R}$, $1\leq
p,q\leq\infty$, $\varepsilon>0$) is the logarithmically modified
Besov space to the standard Besov space $B^{s}_{p,q}$. We also prove
that this system is well-posed for small initial data in
$({B}^{-1}_{\infty,1}\cap{B^{-1,1}_{\infty,\infty}})\times({B}^{-1}_{\frac{n}{2},1}\cap{B^{-1,1}_{\frac{n}{2},\infty}})$.


\end{abstract}
\maketitle
\section{Introduction}


 In this paper we will discuss the Cauchy problem for the normalized
 $n$-dimensional viscous Boussinesq system which describes the natural convection in a viscous
 incompressible fluid as follows:
\begin{align}
 {\vec{u}}_t+(\vec{u}\cdot\nabla)\vec{u}+\nabla P&=\Delta\vec{u}+\theta{\vec{a}} \hskip2cm
 \text{  in } \mathbb{R}^{n}\times (0,\infty),\label{eq4.1.1}\\
 \diva\vec{u}&=0 \hskip3.21cm\text{ in }\mathbb{R}^n\times(0,\infty),\label{eq4.1.2}\\
 \theta_t+\vec{u}\cdot\nabla\theta&=\Delta \theta \hskip2.89  cm\text{ in }\mathbb{R}^n\times (0,\infty),\label{eq4.1.3}
 \\(\vec{u}(\cdot,t),\theta(\cdot,t))|_{t=0}&=(\vec{u}_0(\cdot),\theta_0(\cdot)) \hskip1.38 cm\text{ in }\mathbb{R}^n,\label{eq4.1.4}
 \end{align}
 where $\vec{u}=(u_1(x,t),u_2(x,t),\cdots,u_n(x,t))\in\mathbb{R}^n$ and $P=P(x,t)\in\mathbb{R}$
 denote the unknown vector  velocity and the unknown
 scalar pressure of the fluid, respectively.  $\theta=\theta(x,t)\in\mathbb R$ denotes the density or the temperature.
 $\theta{\vec{a}}$  in \eqref{eq4.1.1} takes into
 account the influence of the gravity and the stratification on the motion of the
 fluid. The whole system is considered under initial condition $(\vec{u}_0,\theta_0)=(\vec{u}_0(x),\theta_0(x))\in\mathbb{R}^{n+1}$.

 The Boussinesq system is extensively used in the atmospheric sciences  and oceanographic turbulence
  (cf. \cite{Majda03}
 and references cited therein). Due to its close relation to fulids, there are a lot of works related to various aspects of this system. Among the fruitful
 results we only cite papers on well-posedness.
 In 1980, Cannon and DiBenedetto in \cite{Cannon80} established well-posedness of the full viscous Boussinesq system in Lebesgue space within the framework of Kato semigroup.
 Around 1990, Mirimoto, Hishida and Kagei have investigated weak
 solutions of this system in  \cite{Morimoto92}, \cite{Hishida91} and
 \cite{Kagei93}. Well-posedness results in pseudomeasure-type space and  weak $L^p$ space,  etc. can be found in \cite{Ferreira08} and
 references cited therein. Recently, the two dimensional Boussinesq system with partial viscous terms has drawn a lot of
attention, see \cite{Abidi07,Chae06,Danchin08,Hmidi10} and
references cited therein.

In this paper, we aim at achieving the lowest regularity results of
the full viscous Boussinesq system with dimension $n\ge2$. Though it
is hard to deal with the coupled term $\vec{u}\nabla\theta$, we
succeed  in finding a suitable product space with regular index
being almost $-1$ in which the Boussinesq system is well-posed. More
precisely, we prove that if $(\vec{u}_0,\theta_0)\in
({B}^{-1}_{\infty,1}\cap{B^{-1,1}_{\infty,\infty}})\times({B}^{-1}_{\frac{n}{2},1}\cap{B^{-1,1}_{\frac{n}{2},\infty}})$
satisfying $\diva\vec{u}_0=0$, where $B^{-1,1}_{p,\infty}$ ($1\leq
p\leq\infty$) is the logarithmically modified Besov space to the
standard Besov space $B^{-1}_{p,\infty}$ (see definition 1.1 below),
then there exists a local solution to Eqs.
 \eqref{eq4.1.1}$\sim$\eqref{eq4.1.4}.
%
We also prove that if $\theta_0$ belongs to $B^{-1}_{p,r}$ with
$p\in(\frac{n}{2},\infty)$ and $r\in[1,\infty]$
  and  $\vec{u}_0$ belongs to ${B}^{-1}_{\infty,1}\cap{B^{-1,1}_{\infty,\infty}}$
  satisfying the divergence free condition, then there exists a local solution to
  Eqs. \eqref{eq4.1.1}$\sim$\eqref{eq4.1.4}.
 The method we use here is essentially frequency localization.

As usual, we use the well-known fixed point arguments and hence we
invert Eqs. \eqref{eq4.1.1} $\sim$ \eqref{eq4.1.4} into the
corresponding integral equations:
\begin{align}
\vec{u}&=e^{t\Delta}\vec{u}_0-\int_{0}^te^{(t-s)\Delta}\mathbb{P}(\vec{u}\cdot\nabla)\vec{u}ds
        +\int_{0}^te^{(t-s)\Delta}\mathbb{P}(\theta{\vec{a}}){ds},\label{eq4.1.5}\\
   \theta&=e^{t\Delta}\theta_0-\int_{0}^te^{(t-s)\Delta}(\vec{u}\cdot\nabla{\theta})ds, \label{eq4.1.6}
  \end{align}
 where $\mathbb{P}$ is the Helmholtz projection operator given by
 $\mathbb{P}=I+\nabla(-\Delta)^{-1}\diva$ with $I$ representing
 the unit operator. In what follows, we shall regard Eqs. \eqref{eq4.1.5}
and \eqref{eq4.1.6} as a fixed point system for the map
$$\mathfrak{J}:\ (\vec{u},\theta)\mapsto
\mathfrak{J}(\vec{u},\theta)=(\mathfrak{J}_{1}(\vec{u},\theta),
\mathfrak{J}_{2}(\vec{u},\theta)),$$ where
$\mathfrak{J}_{1}(\vec{u},\theta)$ and
$\mathfrak{J}_{2}(\vec{u},\theta)$ denote the right-hand sides of
(1.5) and (1.6), respectively.

\vskip.1cm Before showing our main results of this paper, let us
first recall the nonhomogeneous littlewood-Paley decomposition by
means of a sequence of operators $(\triangle_j)_{j\in\mathbb{Z}}$
and then we define the Besov type space $B^{s,\alpha}_{p,r}$ and the
corresponding Chemin-Lerner type space
$\tilde{L}^{\rho}(B^{s,\alpha}_{p,r})$.

To this end, let $\gamma>1$ and $(\varphi,\chi)$ be a couple of smooth functions
valued in $[0,1]$, such that $\varphi$ is supported in the shell
$\{\xi\in\mathbb{R}^n; \gamma^{-1}\le|\xi|\le2\gamma\}$,\ $\chi$ is
supported in the ball $\{\xi\in\mathbb{R}^n; |\xi|\le\gamma\}$ and
\begin{align*}
\chi(\xi)+\sum_{q\in\mathbb{N}}\varphi(2^{-q}\xi)=1,\quad \forall\
\xi\in\mathbb{R}^n.
\end{align*}
For $u\in \mathcal{S}'(\mathbb{R}^n)$, we define nonhomogeneous
dyadic blocks as follows:
\begin{align*}
&\triangle_qu:=0 \quad\text{ if } q\le -2,\\
&\triangle_{-1}u:=\chi(D)u=\tilde{h}\ast{u} \quad\text{ with } \tilde{h}:=\mathcal{F}^{-1}\chi,\\
&\triangle_qu:=\varphi(2^{-q}D)u=2^{qn}\int{h}(2^{q}y)u(x-y)dy
\quad\text{ with } h:=\mathcal{F}^{-1}\varphi \text{ if }q\ge0.
\end{align*}
One can prove that
\begin{align*}
u=\sum_{q\ge-1}\triangle_qu\quad\text{ in }
\mathcal{S}'(\mathbb{R}^n)
\end{align*}
for all tempered distribution $u$. The right-hand side is called
nonhomogeneous Littlewood-Paley decomposition of $u$. It is also
convenient to introduce the following partial sum operator:
\begin{align*}
S_qu := \sum_{ p\le{q-1}} \triangle_pu.
\end{align*}
Obviously we have $S_0u = \triangle_{-1}u$. Since $\varphi(\xi) =
\chi(\xi/2) - \chi(\xi)$ for all $\xi\in \mathbb{R}^n$, one can
prove that
 \begin{align*}S_qu = \chi(2^{-q}D)u = \int\tilde{h}(2^qy)u(x-y)dy\quad\text{ for all } q\in\mathbb{N}.\end{align*}
Let $\gamma
= 4/3$. Then we have the following result, i.e. for any $u\in
\mathcal{S}'(\rn)$ and $v \in \mathcal{S}'(\rn)$, there holds
\begin{align}
 \triangle_k\triangle_qu&\equiv 0 \quad\text{ for } |k-q|\ge2\label{eq4.1.8},\\
\triangle_k(S_{q-1}u\triangle_qv)&\equiv 0 \quad\text{ for }
|k-q|\ge5\label{eq4.1.9},\\
\triangle_k(\triangle_qu\triangle_{q+l}v)&\equiv 0  \quad\text{ for
} |l|\le 1,\;\; k\ge q+4.\label{eq4.1.10}
\end{align}

\vskip.07cm
\begin{definition}\label{def4.1.1} Let $T>0$, $-\infty<s<\infty$ and $1\le p$,
$r$, $\rho\le\infty$.

\textsc{(1)}\ \ We say that a tempered distribution $f\in
B^{s,\alpha}_{p,r}$ if and only if
\begin{align}\label{eq4.1.11}
\Big(\sum_{q\ge-1}2^{qrs}(3+q)^{\alpha{r}}\|\triangle_qf\|_{p}^r\Big)^{\frac{1}{r}}<\infty
\end{align}
(with the usual convention for $r=\infty$).

\textsc{(2)}\ \ We say that a tempered distribution $u\in
\tilde{L}^\rho_T(B^{s,\alpha}_{p,r})$ if and only if
\begin{align}
\label{eq4.1.12}
 \|u\|_{\tilde{L}^{\rho}_T(B^{s,\alpha}_{p,r})}&:=
 \Big(\sum_q2^{qrs}(3+q)^{\alpha r }\|\triangle_qu\|^r_{L^\rho(0,T;L^p_x)}\Big)^{\frac{1}{r}}<\infty.
\end{align}
\end{definition}

\textbf{\large{Remarks.}}\ (i)\ \ The definition (1) is essentially
due to Yoneda \cite{Yoneda10} where he considered the homogeneous
version of the space $B^{s,\alpha}_{p,r}$ (see also remarks there).
Note that by using the heat semigroup characterization of these
spaces (see Lemma \ref{lem4.4.1} in Section 4), we see that
$B^{-1,1}_{\infty,\infty}$ coincides with the space
$B^{-1(\ln)}_{\infty,\infty}$ considered by the second author in his
recent work \cite{Cui11}. The definition (2) in the case $\alpha=0$
(note that $B^{s,0}_{q,r}=B^{s}_{q,r}$) is due to
 Chermin etc. (cf. \cite{Chemin99,Danchin05}).

(ii)\ \ Similar to the case $\alpha=0$ (see \cite{Danchin05} and
references cited therein), by using the Minkowski inequality we
see that for $0\le\alpha\le\beta<\infty$,
\begin{align*}
\|f\|_{\tilde{L}^\rho_T(B^{s,\alpha}_{p,r})}\le\|f\|_{{L}^\rho_TB^{s,\beta}_{p,r}}
\ \ \text{if }r\ge\rho, \quad
\|f\|_{{L}^\rho_TB^{s,\alpha}_{p,r}}\!\le\!\|f\|_{\tilde{L}^\rho_T(B^{s,\beta}_{p,r})}
\ \ \text{if }r\le\rho.
\end{align*}

\medskip\medskip
We now state the main results. In the first two results we consider
the case where the first component $\vec u_0$ of the initial data
lies in the space $B^{-1}_{\infty,1}\cap B^{-1,1}_{\infty,\infty}$.
In the third result we consider the case where $\vec u_0$ lies in
the less regular space $B^{-1,1}_{\infty,\infty}$. As we shall see,
in this case we need the second component $\theta_0$ of the initial
data to lie in a more regular space.

\begin{theorem}\label{thm4.1.2}
Let $n\ge2$. Given  $T>0$, there exist $\mu_1$, $\mu_2>0$ such
that for any $(\vec{u}_0,\theta_0)\in (B^{-1}_{\infty,1}\cap
B^{-1,1}_{\infty,\infty})\times (B^{-1}_{\frac{n}{2},1}\cap
B^{-1,1}_{\frac{n}{2},\infty})$ satisfying
\begin{equation*}
\left\{\begin{aligned}
&\|\vec{u}_0\|_{B^{-1}_{\infty,1}}+\|\vec{u}_0\|_{{B^{-1,1}_{\infty,\infty}}}
\le \mu_1 \ \ \hbox { and }\ \ \diva\vec{u}_0=0,
\\ &\|\theta_0\|_{B^{-1}_{{\frac{n}{2}},1}}+
\|\theta_0\|_{{B^{-1,1}_{{\frac{n}{2}},\infty}}}\le \mu_2,
\end{aligned}\right.
\end{equation*}
 the Boussinesq system  has a unique solution $(\vec{u},\theta)$ in
$\Big(\tilde{L}^{2}_T(B^0_{\infty,1})\cap\tilde{L}^2_T(B^{0,1}_{\infty,\infty})\Big)\times
\Big(\tilde{L}^{2}_T(B^0_{{\frac{n}{2}},1})\cap
\tilde{L}^2_T(B^{0,1}_{{\frac{n}{2}},\infty})\Big)$ and
$ C_w\Big([0,T]; B^{-1}_{\infty,1}\cap B^{-1,1}_{\infty,\infty}\Big)
\times C\Big([0,T];B^{-1}_{\frac{n}{2},1}\cap B^{-1,1}_{\frac{n}{2},\infty}\Big)$
  satisfying
\begin{equation*}
\|\vec{u}\|_{\tilde{L}^{2}_T(B^0_{\infty,1})}+
\|\vec{u}\|_{\tilde{L}^{2}_T(B^{0,1}_{\infty,\infty})}\le
2\mu_1\quad \text{ and }\quad
\|\theta\|_{\tilde{L}^{2}_T(B^0_{{\frac{n}{2}},1})}+
\|\theta\|_{\tilde{L}^{2}_T(B^{0,1}_{{\frac{n}{2}},\infty})}\le
2\mu_2.\end{equation*}
\end{theorem}

\vskip.07cm
\begin{theorem}\label{thm4.1.3}
Let $n\ge2$, $1\le r\le\infty$ and $p\in(\frac{n}{2},\infty)$. Given
$T>0$, there exist $\mu_1$, $\mu_2>0$  such that for any
$(\vec{u}_0,\theta_0)\in (B^{-1}_{\infty,1}\cap
B^{-1,1}_{\infty,\infty})\times B^{-1}_{p,r}$ satisfying
\begin{equation*}
\left\{\begin{aligned}
\|\vec{u}_0\|_{B^{-1}_{\infty,1}}+\|\vec{u}_0\|_{{B^{-1,1}_{\infty,\infty}}}
&\le \mu_1\ \ \hbox{ and }\ \ \diva\vec{u}_0=0,
\\ \|\theta_0\|_{{B^{-1}_{p,r}}}&\le \mu_2,
\end{aligned}\right.
\end{equation*}
the Boussinesq system  has a unique solution $(\vec{u},\theta)$ in
$\Big(\tilde{L}^{2}_T(B^0_{\infty,1})\cap\tilde{L}^2_T(B^{0,1}_{\infty,\infty})\Big)\times
\tilde{L}^{2}_T(B^0_{p,r})$ and $ C_w([0,T]; B^{-1}_{\infty,1}\cap
B^{-1,1}_{\infty,\infty})\times C([0,T];B^{-1}_{p,r\ne\infty})$ or
$C_w([0,T]; B^{-1}_{\infty,1}\cap B^{-1,1}_{\infty,\infty})\times
C_w([0,T];B^{-1}_{p,\infty})$
 satisfying
\begin{align*}
\|\vec{u}\|_{\tilde{L}^{2}_T(B^0_{\infty,1})}+\|\vec{u}\|_{\tilde{L}^{2}_T(B^{0}_{\infty,\infty})}\le
2\mu_1\quad\hbox{ and } \quad
\|\theta\|_{\tilde{L}^{2}_T(B^{0,1}_{p,r})}\le 2\mu_2.
\end{align*}\end{theorem}

\vskip.07cm
\begin{theorem}\label{thm4.1.4} Let $n\ge2$,
$\varepsilon>0$ and $p\in(\frac{n}{2},\infty)$. Given $0<T\le 1$,
there exist $\mu_1=\mu_1(\varepsilon)$, $\mu_2=\mu_2(\varepsilon)>0$
such that for any $(\vec{u}_0,\theta_0)\in
B^{-1,1}_{\infty,\infty}\times B^{-1,\varepsilon}_{p,\infty}$
satisfying
\begin{equation*}
\left\{\begin{aligned} \|\vec{u}_0\|_{{B^{-1,1}_{\infty,\infty}}}
&\le \mu_1\ \ \hbox{ and }\ \ \diva\vec{u}_0=0,
\\ \|\theta_0\|_{{B^{-1,\varepsilon}_{p,\infty}}}&\le \mu_2,
\end{aligned}\right.
\end{equation*}
the Boussinesq system  has a unique solution $(\vec{u},\theta)$ in $ C_w([0,T]; B^{-1,1}_{\infty,\infty})
\times C_w([0,T];B^{-1,\varepsilon}_{p,\infty})$
 satisfying
\begin{equation*}
\sup_{0<t<T}t^{\frac{1}{2}}|\ln(\frac{t}{e^2})|\|\vec{u}\|_\infty\le
2\mu_1 \quad \hbox{ and } \quad
\sup_{0<t<T}t^{\frac{1}{2}}|\ln(\frac{t}{e^2})|^\varepsilon\|\theta\|_{p}\le
2\mu_2.
\end{equation*}
\end{theorem}

\medskip
Later on, we shall use $C$ and $c$ to denote positive constants
which depend on dimension $n$,  $|\vec{a}|$ and might depend on $p$
and may change from line to line.
 $\mathcal{F}f$ and $\hat{f}$  stand for Fourier
transform of $f$ with respect to space variable and
$\mathcal{F}^{-1}$ stands for the inverse Fourier transform.  We
denote $A\le{CB}$ by $A\lesssim B$ and $A\lesssim{B}\lesssim{A}$
  by $A\sim{B}$. For any $1\le{\rho, q}\le\infty$, we denote
 $L^q(\rn)$, $L^\rho(0,T)$ and $L^\rho(0,T;L^q(\rn))$ by  $L^q_x$, $L^\rho_T$
 and $L^\rho_TL^q_x$, respectively. We denote $\|f\|_{L^p_x}$ by $\|f\|_p$ for short. In what follows we will not distinguish  vector valued function space
 and scalar function space if there is no confusion.

We use two different methods which are used by Chermin, etc. and
 Kato, etc. respectively to prove Theorems
\ref{thm4.1.2}$\sim$\ref{thm4.1.3} and Theorem \ref{thm4.1.4}.
Therefore we write their proofs in separate sections.
 In Sect. \!2 we introduce the
paradifferential calculus results, while
 in Sect. \!3 we prove Theorems \ref{thm4.1.2}$\sim$\ref{thm4.1.3}. Finally, in Sect. 4 we prove Theorem \ref{thm4.1.4}.
 \section{Paradifferential calculus}
In this section, we prove several preliminary results concerning
the paradifferential calculus. We first recall some fundamental
results.

\vskip.07cm
\begin{lemma}\label{lem4.2.1}\textsc{(Bernstein)}\ Let
$k$ be in $\mathbb{N}\cup\{0\}$ and $0 < R_1 <R_2$.
There exists a constant $C$ depending only on $R_1, R_2$ and
dimension $n$, such that for all $1 \le a \le b \le \infty$ and
$u\in L^a_x$, we have
\begin{align}\label{eq4.2.1}
\supp \hat{u} \subset B(0,R_1\lambda),\quad
\sup_{|\alpha|=k}\|\partial^\alpha{u}\|_b&\le C^{k+1}\lambda^{k+n(1/a-1/b)} \|u\|_a,\\
\supp \hat{u} \subset C(0,R_1\lambda, R_2\lambda),\quad
\sup_{|\alpha|=k}\|\partial^\alpha{u}\|_a&\sim C^{k+1}\lambda^{k}
\|u\|_a. \label{eq4.2.2}\end{align}
\end{lemma}

\vskip.07cm
\begin{lemma}\label{lem4.2.2}\textsc{\cite{Danchin05, Triebel78}}
 \ Let $1\le p<\infty$. Then we have the following assertions:

\textsc{(}1\textsc{)}\ \  $B^{0}_{\infty,1}\hookrightarrow
\mathcal{C}\cap{L}^{\infty}_x\hookrightarrow
L^{\infty}_x\hookrightarrow B^{0}_{\infty,\infty}$.

\textsc{(}2\textsc{)}\ \  $B^{0}_{p,1}\hookrightarrow
L^p_x\hookrightarrow B^0_{p,\infty}$.
\end{lemma}

\vskip.07cm
\begin{lemma}\label{lem4.2.3}
\textsc{(1)}\  Let $1\le p\le\infty$, $0\le\beta\le\alpha<\infty$,
$-\infty<s<\infty$ and $1\le r_1<r_2\le\infty$. Then
\begin{align*}
  B^{s,\alpha}_{p,r_1}\hookrightarrow B^{s,\alpha}_{p,r_2}\hookrightarrow
  B^{s,\beta}_{p,r_2}.
\end{align*}
\textsc{(2)}\  Let $1<\tilde r\le\infty$, $1\le p\le \infty$ and
$-\infty<s<\infty$. For any $\epsilon>0$, we have
\begin{align}\label{eq4.2.3}
B^{s+\epsilon}_{p,\infty}\hookrightarrow
B^{s,1}_{p,\infty}\hookrightarrow{B^{s}_{p,\tilde r}},\quad
B^{s+\epsilon}_{p,\infty}\hookrightarrow{B^{s}_{p,1}}\hookrightarrow{B^{s}_{p,\tilde
r}}.
\end{align}\\
\textsc{(3)}\  There is no  inclusion relation between the spaces
$B^{0}_{\infty,1}$ and $B^{0,1}_{\infty,\infty}$.
\end{lemma}

\begin{proof}
It suffices to prove $(3)$.\ Similar to \cite{Yoneda10}, we set
\begin{align*}
{f}=\sum_{j=-1}^\infty a_j\delta_{2^j} \text{\ for \
}\{a_j\}_{j=-1}^\infty\subset\mathbb{R},
\end{align*}
where $\delta_z$ is the Dirac delta function massed at
$z\in\mathbb{R}^n$. Then we have
\begin{align*}
\|f\|_{B^{0}_{\infty,1}}\simeq\sum_{j\ge-1}|a_j|, \quad
\|f\|_{B^{0,1}_{\infty,\infty}}\simeq\sup_{j\ge-1}(j+3)|a_j|.
\end{align*}
So if we take $a_j=(j+3)^{-1}$ for $j\ge0$ and $a_j=0$ for $j<0$,
then $\|f\|_{B^{0,1}_{\infty,\infty}}\simeq1$,
 $\|f\|_{B^{0}_{\infty,1}}=\infty$. Therefore,
$B^{0,1}_{\infty,\infty}$ is not included in $B^{0}_{\infty,1}$.
Next, let $\delta_{jk}$ be Kronecker's delta. For fixed
$k\in\mathbb{N}$, if we take $a_j=\frac{\delta_{kj}}{3+j}$ for
$j\ge0$ and $a_j=0$ for $j<0$, then we have
$\|f\|_{B^{0,1}_{\infty,\infty}}\simeq1$ and
$\|f\|_{B^{0}_{\infty,1}}=\frac{1}{k}$. Since $k$ is arbitrary,
$B^{0}_{\infty,1}$ is not included in $B^{0,1}_{\infty,\infty}$.
\end{proof}

We now begin our discussion on paradifferential calculus.

\vskip.07cm
\begin{lemma}\label{lem4.2.4} For any $p,p_1,p_2,r\in [1,\infty]$
satisfying $\frac{1}{p}=\frac{1}{p_1}+\frac{1}{p_2}$,
the bilinear map $(u,v)\mapsto uv$ is bounded from
$(B^{0}_{p_1,1}\cap B^{0,1}_{p_1,\infty})\times B^{0}_{p_2,r}$ to
$B^{0}_{p,r}$, i.e.
\begin{align}\label{eq4.2.4}
\|uv\|_{B^{0}_{p,r}}\lesssim \|u\|_{B^{0}_{p_1,1}\cap
B^{0,1}_{p_1,\infty}}\|v\|_{B^{0}_{p_2,r}}.
\end{align}
\end{lemma}

\begin{proof} Following Bony \cite{Bony81} we write
\begin{align*}
uv=\mathcal{T}(u,v)+\mathcal{T}(v,u)+\mathcal{R}(u,v),
\end{align*}
where
\begin{align*}
\mathcal{T}(u,v)=\sum_{q\ge-1}S_{q-1}u\triangle_qv, \quad
\mathcal{R}(u,v)=\sum_{l=-1}^1\sum_{q\ge-1}\triangle_qu\triangle_{q+l}v.
\end{align*}
The estimate of $\mathcal{T}(u,v)$ is simple. Indeed, by Proposition
1.4.1 (i) of \cite{Danchin05} we know that for any $p,r\in
[1,\infty]$, $\mathcal{T}$ is bounded from
$L^\infty_x\times{B^{0}_{p,r}}$ to $B^{0}_{p,r}$. By slightly
modifying the proof of that proposition, we see that for any
$p,p_1,p_2,r\in [1,\infty]$ satisfying
$\frac{1}{p}=\frac{1}{p_1}+\frac{1}{p_2}$, $\mathcal{T}$ is also
bounded from $L^{p_1}_x\times{B^{0}_{p_2,r}}$ to $B^{0}_{p,r}$. Thus
for any $p,p_1,p_2,r\in [1,\infty]$ we have
\begin{align}\label{eq4.2.5}
\|\mathcal{T}(u,v)\|_{B^{0}_{p,r}}\lesssim
\|u\|_{L^{p_1}}\|v\|_{B^{0}_{p_2,r}}\lesssim
\|u\|_{B^{0}_{p_1,1}}\|v\|_{B^{0}_{p_2,r}}.
\end{align}
In what follows we estimate $\mathcal{T}(v,u)$ and
$\mathcal{R}(u,v)$. By interpolation, it suffices to consider the
two end point cases $r=1$ and $r=\infty$.

To estimate $\|\mathcal{T}(v,u)\|_{B^{0}_{p,1}}$, we use
\eqref{eq4.1.9} to deduce
\begin{align}\label{eq4.2.6}
 \|\mathcal{T}(v,u)\|_{B^{0}_{p,1}}=&\sum_{k\ge-1}
 \|\triangle_k(\sum_{q\ge-1}S_{q-1}v\triangle_qu)\|_p
 \lesssim \sum_{k\ge-1}\sum_{|q-k|\le4,q\ge-1}\|S_{q-1}v\|_{p_2}\|\triangle_{q}u\|_{p_1}\nonumber\\
 \lesssim &\|v\|_{p_2}\sum_{k\ge-1}\|\triangle_{k}u\|_{p_1}
 \lesssim \|u\|_{B^{0}_{p_1,1}} \|v\|_{B^0_{p_2,1}}.
\end{align}
To estimate $\|\mathcal{T}(v,u)\|_{B^{0}_{p,\infty}}$ we first note
that
\begin{align*}
 \|S_{q-1}v\|_{p_2}\lesssim &\sum_{j=-1}^{q-2}\|\triangle_jv\|_{p_2}
 \lesssim (q+3)\|v\|_{B^0_{p_2,\infty}}.
\end{align*}
Using this inequality and \eqref{eq4.1.9} we see that
\begin{align}\label{eq4.2.7}
 \|\mathcal{T}(v,u)\|_{B^{0}_{p,\infty}}=&\sup_{k\ge-1}
 \|\triangle_k(\sum_{q\ge-1}S_{q-1}v\triangle_qu)\|_p
 \lesssim \sup_{k\ge-1}\sum_{|q-k|\le4,q\ge-1}\|S_{q-1}v\|_{p_2}\|\triangle_{q}u\|_{p_1}\nonumber\\
 \lesssim &\|v\|_{B^0_{p_2,\infty}}\sup_{k\ge-1}\sum_{|q-k|\le4,q\ge-1}(q+3)\|\triangle_{q}u\|_{p_1}
 \lesssim \|u\|_{B^{0,1}_{p_1,\infty}}\|v\|_{B^0_{p_2,\infty}}.
\end{align}
To estimate $\|\mathcal{R}(u,v)\|_{B^{0}_{p,1}}$, we write
\begin{align*}
 \|\mathcal{R}(u,v)\|_{B^{0}_{p,1}}\lesssim &\sum_{l=-1}^1\sum_{k\ge-1}
 \|\triangle_k(\sum_{q\ge-1}\triangle_qu\triangle_{q+l}v)\|_p\nonumber\\
 \lesssim &\sum_{l=-1}^1\sum_{k=-1}^3
 \|\triangle_{k}(\sum_{q\ge-1}\!\!\triangle_qu\triangle_{q+l}v)\|_{p}+
 \sum_{l=-1}^1\sum_{k\ge4}
 \|\triangle_{k}(\sum_{q\ge-1}\triangle_qu\triangle_{q+l}v)\|_{p}\nonumber\\
 :=& I_{1}+I_{2}.
\end{align*}
For $I_1$ we have
\begin{align*}
 I_{1}&\lesssim\sum_{l=-1}^1\|\sum_{q\ge-1}\triangle_qu\triangle_{q+l}v\|_{p}
 \lesssim\sum_{l=-1}^1\sum_{q\ge-1}\|\triangle_qu\|_{p_1}
 \|\triangle_{q+l}v\|_{p_2}\lesssim
 \|u\|_{B^{0}_{p_1,\infty}}\|v\|_{B^{0}_{p_2,1}}.
\end{align*}
For $I_2$, by using \eqref{eq4.1.10} we deduce
\begin{align*}
  I_{2}\lesssim&\sum_{l=-1}^1\sum_{k\ge4}\sum_{q\ge{k-3}}\|\triangle_qu\|_{p_1}\|\triangle_{q+l}v\|_{p_2}
 \lesssim\sum_{l=-1}^1\sum_{q\ge1}\sum_{4\le k\le3+q}\|\triangle_qu\|_{p_1}\|\triangle_{q+l}v\|_{p_2}
 \nonumber\\
 \lesssim&\sum_{l=-1}^1\sum_{q\ge1}(q+3)\|\triangle_qu\|_{p_1}\|\triangle_{q+l}v\|_{p_2}
 \lesssim \|u\|_{B^{0,1}_{p_1,\infty}}\|v\|_{B^{0}_{p_2,1}}.
\end{align*}
Hence
\begin{align}\label{eq4.2.8}
 \|\mathcal{R}(u,v)\|_{B^{0}_{p,1}}\lesssim\|u\|_{B^{0,1}_{p_1,\infty}}\|v\|_{B^{0}_{p_2,1}}.
\end{align}
Similarly we have
\begin{align}\label{eq4.2.9}
    \|\mathcal{R}(u,v)\|_{B^{0}_{p,\infty}}
  &\lesssim\sum_{l=-1}^1\sup_{k\ge-1}\|\triangle_k(\sum_{q\ge-1}\triangle_qu\triangle_{q+l}v)\|_{p}
  \nonumber\\&\lesssim\sum_{l=-1}^1\Big(\sum_{q\ge-1}\|\triangle_qu\|_{p_1}\|\triangle_{q+l}v\|_{p_2}\Big)
  \nonumber\\&\lesssim\|u\|_{B^{0}_{p_1,1}}\|v\|_{B^{0}_{p_2,\infty}}.
\end{align}
From \eqref{eq4.2.5}$\sim$\eqref{eq4.2.9} and interpolation, we
obtain the desired estimate. This completes the proof of Lemma
\ref{lem4.2.4}.
\end{proof}

\vskip.07cm
\begin{lemma}\label{lem4.2.5}
For any $p,p_2,p_2\in [1,\infty]$ satisfying
$\frac{1}{p}=\frac{1}{p_1}+\frac{1}{p_2}$, the bilinear map
$(u,v)\mapsto uv$ is bounded from $(B^{0}_{p_1,1}\cap
B^{0,1}_{p_1,\infty})\times (B^{0}_{p_2,1}\cap
B^{0,1}_{p_2,\infty})$ to $B^{0}_{p,1}\cap B^{0,1}_{p,\infty}$,
i.e.
\begin{align}\label{eq4.2.10}
\|uv\|_{B^{0}_{p,1}\cap B^{0,1}_{p,\infty}}\lesssim
\|u\|_{B^{0}_{p_1,1}\cap B^{0,1}_{p_1,\infty}}
\|v\|_{B^{0}_{p_2,1}\cap B^{0,1}_{p_2,\infty}}.
\end{align}
In particular, $B^{0}_{\infty,1}\cap B^{0,1}_{\infty,\infty}$ is a
Banach algebra.
\end{lemma}

\begin{proof}
By Lemma \ref{lem4.2.4}, we only need to prove that
\begin{align}\label{eq4.2.11}
\|uv\|_{B^{0,1}_{p,\infty}}\lesssim \|u\|_{B^{0}_{p_1,1}\cap
B^{0,1}_{p_1,\infty}}
\|v\|_{B^{0}_{p_2,1}\cap{B}^{0,1}_{p_2,\infty}}.
\end{align}
As before we decompose $uv$ into the sum of $\mathcal{T}(u,v)$,
$\mathcal{T}(v,u)$ and $\mathcal{R}(u,v)$. To estimate
$\|\mathcal{T}(u,v)\|_{B^{0,1}_{p,\infty}}$, we use \eqref{eq4.1.9}
to deduce
\begin{align}\label{eq4.2.12}
 \|\mathcal{T}(u,v)\|_{B^{0,1}_{p,\infty}}=&\sup_{k\ge-1}
 (k+3)\|\triangle_k(\sum_{q\ge-1}S_{q-1}u\triangle_qv)\|_p\nonumber\\
 \lesssim &\sup_{k\ge-1}(k+3)\|\triangle_k(\sum_{|q-k|\le 4,q\ge-1}S_{q-1}u\triangle_qv)\|_p\nonumber\\
 \lesssim &\sup_{k\ge-1}(k+3)\sum_{|q-k|\le4,q\ge-1}\|S_{q-1}u\|_{p_1}\|\triangle_{q}v\|_{p_2}\nonumber\\
 \lesssim &\|u\|_{p_1}\sup_{q\ge-1}(q+3)\|\triangle_{q}v\|_{p_2}\nonumber\\
 \lesssim &\|u\|_{B^0_{p_1,1}}\|v\|_{B^{0,1}_{p_2,\infty}}.
\end{align}
The estimate of $\|\mathcal{T}(v,u)\|_{B^{0,1}_{p,\infty}}$ is
similar, with minor modifications. Indeed,
\begin{align}\label{eq4.2.13}
 \|\mathcal{T}(v,u)\|_{B^{0,1}_{p,\infty}}=&\sup_{k\ge-1}
 (k+3)\|\triangle_k(\sum_{q\ge-1}S_{q-1}v\triangle_qu)\|_p\nonumber\\
 \lesssim &\sup_{k\ge-1}(k+3)\sum_{|q-k|\le4,q\ge-1}\|S_{q-1}v\|_{p_2}\|\triangle_{q}u\|_{p_1}\nonumber\\
 \lesssim &\|v\|_{p_2}\sup_{q\ge-1}(q+3)\|\triangle_{q}u\|_{p_1}\nonumber\\
 \lesssim &\|u\|_{B^{0,1}_{p_1,\infty}}\|v\|_{B^0_{p_2,1}}.
\end{align}
To estimate $\|\mathcal{R}(u,v)\|_{B^{0,1}_{p,\infty}}$ we write
\begin{align*}
\|\mathcal{R}(u,v)\|_{B^{0,1}_{p,\infty}}\!&\!\le\!\!\sum_{l=-1}^1
\sup_{k\ge-1}(3+k)\|\triangle_k(\sum_{q\ge-1}\triangle_qu\triangle_{q+l}v)\|_{p}
\\&\le\sum_{l=-1}^1\sup_{7\ge k\ge-1}(3+k)\|\triangle_k(\sum_{q\ge-1}\triangle_qu\triangle_{q+l}v)\|_{p}
\\&\quad+\sum_{l=-1}^1\sup_{q+3\ge
k\ge8}(3+k)\|\triangle_k(\sum_{q\ge5}\triangle_qu\triangle_{q+l}v)\|_{p}\\:=&I_{3}+I_{4}.
\end{align*}
For $I_{3}$ we have
\begin{align*}
 I_{3}&\lesssim
 \sum_{l=-1}^1\sum_{q\ge-1}\|\triangle_qu\|_{p_1}\|\triangle_{q+l}v\|_{p_2}
 \lesssim\|u\|_{B^{0}_{p_1,1}}\|v\|_{B^{0}_{p_2,\infty}}
 \lesssim\|u\|_{B^{0}_{p_1,1}}\|v\|_{B^{0,1}_{p_2,\infty}}.
\end{align*}
For $I_{4}$ we have
\begin{align*}
I_{4}&=\sum_{l=-1}^1\sup_{q+3\ge k\ge8}(3+k)\|\triangle_k(\sum_{q\ge5}\triangle_qu\triangle_{q+l}v)\|_{p}\\
&\lesssim\sum_{l=-1}^1\sum_{q\ge4}(3+q)\|\triangle_qu\|_{p_1}\|\triangle_{q+l}v\|_{p_2}
\lesssim\|u\|_{B^{0}_{p_1,1}}\|v\|_{B^{0,1}_{p_2,\infty}}.
\end{align*}
Hence
\begin{align}\label{eq4.2.14}
 \|\mathcal{R}(u,v)\|_{B^{0,1}_{p,\infty}}\lesssim\|u\|_{B^{0}_{p_1,1}}\|v\|_{B^{0,1}_{p_2,\infty}}.
\end{align}
Combining \eqref{eq4.2.12}$\sim$\eqref{eq4.2.14}, we see that
\eqref{eq4.2.11} follows. This prove Lemma \ref{lem4.2.5}.
\end{proof}



\begin{lemma}\label{lem4.2.6}
Let $p$, $p_i$, $r$, $\rho$, $\rho_i\in[1,\infty]$ $(i=1, 2)$ be
such that
 $\frac{1}{\rho}=\frac{1}{\rho_1}+\frac{1}{\rho_2}$ and $\frac{1}{p}=\frac{1}{p_1}+\frac{1}{p_2}$.
 Then we have
\begin{align}\label{eq4.2.15}
 \|uv\|_{\tilde{L}^\rho_T(B^0_{p,r})}\lesssim\|u\|_{\tilde{L}^{\rho_1}_T(B^0_{p_1,1})
 \cap\tilde{L}^{\rho_1}_T(B^{0,1}_{p_1,\infty})}
 \|v\|_{\tilde{L}^{\rho_2}_T(B^0_{p_2,r})}.
\end{align}
\end{lemma}

\begin{proof}
The proof is similar to that of Lemma \ref{lem4.2.4}. Indeed, as
in the proof of Lemma \ref{lem4.2.4} we decompose $uv$ into the
sum of $\mathcal{T}(u,v)$, $\mathcal{T}(v,u)$ and
$\mathcal{R}(u,v)$. To estimate
$\|\mathcal{T}(u,v)\|_{\tilde{L}^\rho_T(B^{0}_{p,\infty})}$, we
use \eqref{eq4.1.9} to write
\begin{align}
 \|\mathcal{T}(u,v)\|_{\tilde{L}^\rho_T(B^{0}_{p,\infty})}=&\sup_{k\ge-1}
 \|\triangle_k(\sum_{q\ge-1}S_{q-1}u\triangle_qv)\|_{L^{\rho}_TL^p_x}\nonumber\\
 \lesssim &\sup_{k\ge-1}\|\triangle_k(\sum_{|q-k|\le 4,q\ge-1}S_{q-1}u\triangle_qv)\|_{L^\rho_TL^p_x}\nonumber\\
 \lesssim &\sup_{k\ge-1}\sum_{|q-k|\le4,q\ge-1}\|S_{q-1}u\|_{L^{\rho_1}_TL^{p_1}_x}\|\triangle_{q}v\|_{L^{\rho_2}_TL^{p_2}_{x}}\nonumber\\
 \lesssim
 &\|u\|_{\tilde{L}^{\rho_1}_T(B^{0}_{p_1,1})}\|v\|_{\tilde{L}^{\rho_2}_T(B^{0}_{p_2,\infty})}\nonumber.
\end{align}
The estimates of
$\|\mathcal{T}(v,u)\|_{\tilde{L}^\rho_T(B^{0}_{p,\infty})}$ and
$\|\mathcal{R}(u,v)\|_{\tilde{L}^\rho_T(B^{0}_{p,\infty})}$ are
similar and we omit the details here.
\end{proof}

\begin{lemma}\label{lem4.2.7}
Let $p$, $p_i$, $\rho$, $\rho_i\in[1,\infty]$ $(i=1, 2)$ be such
that $\frac{1}{\rho}=\frac{1}{\rho_1}+\frac{1}{\rho_2}$ and
$\frac{1}{p}=\frac{1}{p_1}+\frac{1}{p_2}$. Then we have
\begin{align}\label{eq4.2.16}
 \|uv\|_{\tilde{L}^\rho_T(B^0_{p,1})\cap\tilde{L}^\rho_T(B^{0,1}_{p,\infty})}\lesssim
 \|u\|_{\tilde{L}^{\rho_1}_T(B^0_{p_1,1})\cap\tilde{L}^{\rho_1}_T(B^{0,1}_{p_1,\infty})}
 \|v\|_{\tilde{L}^{\rho_2}_T(B^0_{p_2,1})\cap\tilde{L}^{\rho_2}_T(B^{0,1}_{p_2,\infty})}.
\end{align}
\end{lemma}
\begin{proof}
The proof is similar to that of Lemma \ref{lem4.2.5}; we thus omit
it.\end{proof}

We note that results obtained in Lemmas
\ref{lem4.2.4}$\sim$\ref{lem4.2.7} still hold for vector valued
functions.

\section{Proofs of Theorems \ref{thm4.1.2} and \ref{thm4.1.3}}

In this section, we give the proofs of Theorems \ref{thm4.1.2} and
\ref{thm4.1.3}. We need the following preliminary result:

\begin{lemma}\label{lem4.3.1}\textsc{(\cite{Cannone03}, \text{p}.189, Lemma
5)} \ Let $(\mathcal{X}\times \mathcal{Y},\ \|\cdot\|_{\mathcal{X}}+
\|\cdot\|_{\mathcal{Y}})$  be an abstract  Banach product  space.
$B_1:\mathcal{X}\times\mathcal{X}\rightarrow\mathcal{X}$,  $B_2:
 \mathcal{X}\times\mathcal{Y}\rightarrow\mathcal{Y}$  and
$L:\mathcal{Y}\rightarrow\mathcal{X}$ are respectively two
bilinear operators and one  linear operator such that for any
$(x_i,y_i)\in \mathcal{X}\times\mathcal{Y}$ ($i=1, 2$), we have
\begin{align*}
 \|B_1(x_1,x_2)\|_{\mathcal{X}}\le {\lambda}\|x_1\|_{\mathcal{X}}\|x_2\|_{\mathcal{X}},\quad
  \|L(y_i)\|_{\mathcal{X}}\le{\eta}\|y_i\|_{\mathcal{Y}},\quad
 \|B_2(x_i,y_i)\|_{\mathcal{Y}}\le{\lambda}\|x_i\|_{\mathcal{X}}\|y_i\|_{\mathcal{Y}},
\end{align*}
 where  $\lambda,\eta>0$. For any
 $(x_0,y_0)\in\mathcal{X}\times\mathcal{Y}$ with
 $\|(x_0,c_\ast{y_0})\|_{\mathcal{X}\times\mathcal{Y}}<1/(16\lambda)$ $(c_\ast=\max\{2\eta,1\})$, the following system
\begin{equation*}\label{eq4.3-1}
     (x, y)=(x_0, y_0)
          +\Big(B_1(x,x),\ B_2(x,y)\Big) +\Big(L(y),\ 0\Big)
       \end{equation*}
 has a   solution $(x,y)$ in $\mathcal{X}\times\mathcal{Y}$. In particular,  the solution is such that
\begin{align*}
 \|(x,c_\ast{y})\|_{\mathcal{X}\times\mathcal{Y}}\le{4\|(x_0,c_\ast{y_0})\|_{\mathcal{X}\times\mathcal{Y}}}
\end{align*}
 and it is the only one such that
$\|(x,c_\ast{y})\|_{\mathcal{X}\times\mathcal{Y}}<{1}/{(4\lambda)}.$
\end{lemma}

For $n\ge2$, $p\in(\frac{n}{2},\infty)$ and $r\in[1,\infty]$, let
$\mathcal{X}_T$ and  $\mathcal{Z}_T$ respectively be the spaces
$$\mathcal{X}_T=\tilde{L}^2_T(B^{0}_{\infty,1})\cap\tilde{L}^2_T(B^{0,1}_{\infty,\infty})\quad \hbox{ and }\quad
\mathcal{Z}_T=\tilde{L}^2_T(B^{0}_{p,r})$$
with norms
 $$\|\vec{u}\|_{\mathcal{X}_T}:=\|\vec{u}\|_{\tilde{L}^2_T(B^{0}_{\infty,1})}+\|\vec{u}\|_{\tilde{L}^2_T(B^{0,1}_{\infty,\infty})}\quad\hbox{ and
  }\quad \|\theta\|_{\mathcal{Z}_T}:=\|\theta\|_{\tilde{L}^2_T(B^{0}_{p,r})}.$$
Let $\mathcal{Y}_T$ be the space
$$\mathcal{Y}_T=\tilde{L}^2_T(B^{0}_{\frac{n}{2},1})\cap\tilde{L}^2_T(B^{0,1}_{\frac{n}{2},\infty})$$
with norm
 $$\|\theta\|_{\mathcal{Y}_T}:=\|\theta\|_{\tilde{L}^2_T(B^{0}_{\frac{n}{2},1})}+\|\theta\|_{\tilde{L}^2_T(B^{0,1}_{\frac{n}{2},\infty})}.$$
Recall that
\begin{align}\label{eq4.3.0}
\left\{\begin{aligned}
 \mathfrak{J}_1(\vec{u},\theta)&=e^{t\Delta}\vec{u}_0-\int_{0}^te^{(t-s)\Delta}\mathbb{P}(\vec{u}\cdot\nabla)\vec{u}ds
        +\int_{0}^te^{(t-s)\Delta}\mathbb{P}(\theta{\vec{a}}){ds},\\
   \mathfrak{J}_2(\vec{u},\theta)&=e^{t\Delta}\theta_0-\int_{0}^te^{(t-s)\Delta}(\vec{u}\cdot\nabla{\theta})ds.
   \end{aligned}\right.
   \end{align}

In what follows, we prove several bilinear estimates.
\begin{lemma}\label{lem4.3.2}
 Let $T>0$, $n\ge2$ and
 $\vec{u}_0\in{B^{-1}_{\infty,1}}\cap{B}^{-1,1}_{\infty,\infty}$. We have the following two assertions:

\textsc{(1)}\ \  For
 $\theta\in\mathcal{Y}_T$ we have
\begin{align}
  \|\mathfrak{J}_1(\vec{u},\theta)\|_{\mathcal{X}_T}\lesssim
  &\ (1+T)\Big(\|\vec{u}_0\|_{B^{-1}_{\infty,1}{\cap}B^{-1,1}_{\infty,\infty}}+
  \|\vec{u}\|^2_{\mathcal{X}_T}+
  \|\theta\|_{\mathcal{Y}_T}\Big).\label{eq4.3.1}
\end{align}

\textsc{(2)}\ \ For $\theta\in\mathcal{Z}_T$, $r\in[1,\infty]$ and
$p\in(\frac{n}{2},\infty)$ we have
\begin{align}
  \|\mathfrak{J}_1(\vec{u},\theta)\|_{\mathcal{X}_T}\lesssim
  &\ (1+T)\Big(\|\vec{u}_0\|_{B^{-1}_{\infty,1}{\cap}B^{-1,1}_{\infty,\infty}}+
  \|\vec{u}\|^2_{\mathcal{X}_T}+
  \|\theta\|_{\mathcal{Z}_T}\Big).\label{eq4.3.2}
\end{align}

\end{lemma}
\begin{proof}
 We divide the proof of the $\mathfrak{J}_1(\vec{u},\theta)$ into two subcases $q\ge0$ and
$q=-1$. Since when $q\ge0$, the symbol of ${\triangle}_{q}$ is
supported in dyadic shells and the symbol of $\mathbb{P}$ is smooth
in the corresponding dyadic shells we have
\begin{align}\label{eq4.3.3}
 \triangle_q\mathfrak{J}_1(\vec{u},\theta)=e^{t\Delta}\triangle_q\vec{u}_0-
 \int_0^te^{(t-\tau)\Delta}{\triangle}_q\left(\mathbb{P}\nabla\cdot(\vec{u}\otimes{\vec{u}})(\tau)-\mathbb{P}(\theta{\vec{a}})
 \right)d\tau.
\end{align}
In \eqref{eq4.3.3} we have $n$ scalar equations and each of the $n$
components shares the same estimate. By making use of
\eqref{eq4.2.2} twice we obtain
\begin{align}\nonumber
 \|\triangle_q\mathfrak{J}_1(\vec{u},\theta)\|_\infty&\lesssim{e^{-\kappa2^{2q}t}}{\|\triangle_q{\vec{u}_0}\|_\infty}
 +\int_0^te^{-\kappa2^{2q}(t-\tau)}\Big(2^q\|\triangle_q(\vec{u}\otimes\vec{u})\|_\infty+ \|\triangle_q\theta\|_{\infty}\Big)d\tau\\
 \nonumber
 &\lesssim{e^{-\kappa2^{2q}t}}{\|\triangle_q{\vec{u}_0}\|_\infty}+\int_0^te^{-\kappa2^{2q}(t-\tau)}2^q\|\triangle_q(\vec{u}\otimes\vec{u})\|_\infty
 d\tau \\ & \ \ \ \
 +\int_0^te^{-\kappa2^{2q}(t-\tau)}\min\{2^{2q}\|\triangle_q\theta\|_{\frac{n}{2}}, \ 2^{\frac{qn}{p}}\|\triangle_q\theta\|_{p}\}d\tau\nonumber.
\end{align}
Applying convolution inequalities to  the above estimate with
respect to time variable  we get
\begin{align}\nonumber
 \|\triangle_q\mathfrak{J}_1(\vec{u},\theta)\|_{L^2_TL^\infty_x}\lesssim&\left(\frac{1-e^{-2\kappa{T}2^{2q}}}{2\kappa}\right)^{\half}
 \left(2^{-q}\|\triangle_q\vec{u}_0\|_\infty +\|\triangle_q(\vec{u}\otimes\vec{u})\|_{L^1_TL^\infty_x}\right)\\&+
  \min\Big\{ \frac{1-e^{-2\kappa{T}2^{2q}}}{2\kappa}
 \|\triangle_q\theta\|_{L^2_TL^{\frac{n}{2}}_x},\ \ \frac{1-e^{-2\kappa{T}2^{2q}}}{2\kappa2^{q(2-\frac{n}{p})}}
 \|\triangle_q\theta\|_{L^2_TL^{p}_x})\Big\}\label{eq4.3.4}.
\end{align}
Considering
$$\sum_{q\ge0}2^{-q(2-\frac{n}{p})}<\infty \ \ \hbox{ for } p\in(\frac{n}{2},\infty)\cap[1,\infty),$$
 from \eqref{eq4.3.4} and Definition \ref{def4.1.1} we see that
\begin{align}\label{eq4.3.5}
\sum_{q\ge0}\|\triangle_q\mathfrak{J}_1(\vec{u},\theta)\|_{L^2_TL^\infty_x}&\lesssim\sum_{q\ge-1}\Big({2^{-q}}{\|\triangle_q\vec{u}_0\|_\infty}
 +\|\triangle_q(\vec{u}\otimes\vec{u})\|_{L^1_TL^\infty_x}\nonumber\\ \nonumber
 &\hskip1cm+\min\{\|\triangle_q\theta\|_{L^2_TL^{\frac{n}{2}}_x},\ \ {2^{-q(\frac{n}{p}-2)}}\|\triangle_q\theta\|_{L^2_TL^{p}_x}\}\Big)\nonumber\\
 & \lesssim\|\vec{u}_0\|_{B^{-1}_{\infty,1}}\!+\!\|\vec{u}\otimes\vec{u}\|_{\tilde{L}^1_T(B^{0}_{\infty,1})}
 \!+\!\min\{\|\theta\|_{\tilde{L}^2_T(B^{0}_{\frac{n}{2},2})},\ \|\theta\|_{\tilde{L}^2_T(B^{0}_{p,\infty})}\}.\end{align}
Considering
 $$\sup_{q\ge-1} \frac{q+3}{2^{q(2-\frac{n}{p})}}\le C(p,n)<\infty\quad \hbox{ for }
 p>\frac{n}{2},$$
 from \eqref{eq4.3.4}, Definition \ref{def4.1.1} and a similar argument as before we see that
\begin{align}
\sup_{q\ge0}(q+3)\|\triangle_q\mathfrak{J}_1(\vec{u},\theta)\|_{L^2_TL^\infty_x}
 \lesssim\|\vec{u}_0\|_{B^{-1,1}_{\infty,\infty}}&+\|\vec{u}\otimes\vec{u}\|_{\tilde{L}^1_T(B^{0,1}_{\infty,\infty})}
 \nonumber\\&\label{eq4.3.6} +\min\{\|\theta\|_{\tilde{L}^2_T(B^{0,1}_{\frac{n}{2},\infty})},\ \|\theta\|_{\tilde{L}^2_T(B^{0}_{p,\infty})}\}.\end{align}
Next we consider the case $q=-1$. We recall the decay estimates of
Oseen kernel (cf., Chapter 11,
 \cite{Le02}), by interpolating we observe that $e^\Delta\mathbb{P}(-\Delta)^{-\half+\delta}\nabla$ (for any
 $\delta\in(0,\half)$) is $L^1_x$ bounded. Similar to
 \eqref{eq4.3.3} we get
\begin{align*}
  S_0\mathfrak{J}_1(\vec{u},\theta)= & e^{t\Delta}S_0\vec{u}_0-\int_0^te^{(t-\tau)\Delta}S_0[\mathbb{P}\nabla\cdot(\vec{u}\otimes\vec{u})(\tau)]d\tau
 +\int_0^te^{(t-\tau)\Delta}S_0\mathbb{P}(\theta\vec{a})(\tau)d\tau.
\end{align*}
Applying decay estimates of heat  kernel and Lemma \ref{lem4.2.1} of
$\supp\mathcal{F}{S_0}\subset B(0,\frac{4}{3})$ we see that
\begin{align*}
 \|S_0\mathfrak{J}_1(\vec{u},\theta)\|_\infty&\lesssim
 \|e^{t\Delta}S_0\vec{u}_0\|_\infty+\int_0^t(t-\tau)^{-\frac{n}{4p}}\|\mathbb{P}S_0\theta\|_{2p}d\tau\\&\ \ \ +
 \int_0^t\|e^{(t-\tau)\Delta}{S}_{0}\mathbb{P}\nabla
 \cdot(\vec{u}\otimes\vec{u})(\tau)\|_\infty
 d\tau\\&\lesssim
 \|S_0\vec{u}_0\|_\infty+\int_0^t(t-\tau)^{-\frac{n}{4p}}\min\{\|S_0\theta\|_{p},\
 \
  \|S_0\theta\|_{\frac{n}{2}}\}d\tau\\&\
 \ \
 +\int_0^t(\|S_0(\vec{u}\otimes\vec{u})\|_\infty d\tau.
\end{align*}
In the above estimate we have used the following fact (see (5.29) of
\cite{Sawada04}):
\begin{align*}
\|S_0\mathbb{P}\nabla\cdot(\vec{u}\otimes\vec{u})\|_{\infty}\lesssim\|\mathbb{P}\nabla\tilde{h}\|_1\|S_0(\vec{u}\otimes\vec{u})
\|_{\infty}\lesssim\|\nabla\tilde{h}\|_{\dot{B}^{0}_{1,1}}\|S_0(\vec{u}\otimes\vec{u})
\|_{\infty}\lesssim\|S_0(\vec{u}\otimes\vec{u}) \|_{\infty}.
\end{align*}
Applying convolution inequalities to time variable we obtain that
\begin{align}\nonumber
\|S_0\mathfrak{J}_1(\vec{u},\theta)\|_{L^2_TL^\infty_x}&\!\lesssim\!
 T^{\half}\|S_0\vec{u}_0\|_\infty\!
  +T^{\frac{1}{2}}\|S_0(\vec{u}\otimes\vec{u})\|_{L^1_TL^\infty_x}+\!T^{\frac{4p-n}{4p}}\min\{
 \|S_0\theta\|_{L^2_TL^{\!\frac{n}{2}}_x},  \|S_0\theta\|_{L^2_TL^{p}_x}\}\\&\label{eq4.3.7}
  \lesssim\! C_T\Big[
  \|\vec{u}_0\|_{B^{-1}_{\infty,\infty}}\!\!\!\!+\!\!\|\vec{u}\otimes\vec{u}\|_{\tilde{L}^1_T(B^0_{\infty,\infty}\!)}\!+\!
  \min\{\|\theta\|_{\tilde{L}^2_T(B^0_{\frac{n}{2}},\infty\!)},
    \|\theta\|_{\tilde{L}^2_T(B^0_{p,\infty}\!)}\}\!
  \Big]\\&
  \lesssim\! C_T\Big[
  \|\vec{u}_0\|_{B^{-1,1}_{\infty,\infty}}\!\!\!\!+\!\!\|\vec{u}\otimes\vec{u}\|_{\tilde{L}^1_T(B^{0,1}_{\infty,\infty}\!)}\!+\!
  \min\{\|\theta\|_{\tilde{L}^2_T(B^{0,1}_{\frac{n}{2}},\infty\!)},
    \|\theta\|_{\tilde{L}^2_T(B^0_{p,\infty}\!)}\}\!
  \Big],\label{eq4.3.8}
\end{align}
where $C_T=(T^{\frac{4p-n}{4p}}+ T^{\frac{1}{2}})$.

By applying  \eqref{eq4.3.5} $\sim$ \eqref{eq4.3.8} and Definition
\ref{def4.1.1} as well as Lemmas \ref{lem4.2.6} $\sim$
\ref{lem4.2.7} we prove \eqref{eq4.3.1} and \eqref{eq4.3.2} and we
complete the proof of
 Lemma \ref{lem4.3.2}.
\end{proof}

\vskip.07cm
\begin{lemma}\label{lem4.3.3}
 Let $T>0$, $n\ge2$
 and $\vec{u}\in\mathcal{X}_T$. We have the following assertions:

 \textsc{(1)}\ \  For
 $\theta_0\in{B^{-1}_{\frac{n}{2},1}}\cap{B}^{-1,1}_{\frac{n}{2},\infty}$
  we have
\begin{align}
  \|\mathfrak{J}_2(\vec{u},\theta)\|_{\mathcal{Y}_T}\lesssim&\
    (1+T )\Big(\|\theta_0\|_{B^{-1}_{\frac{n}{2},1}{\cap}B^{-1,1}_{\frac{n}{2},\infty}}+
  \|\vec{u}\|_{\mathcal{X}_T}\|\theta\|_{\mathcal{Y}_T}\Big).\label{eq4.3.9}
\end{align}

 \textsc{(2)}\ \ For $\theta_0\in{B^{-1}_{p,r}}$, $r\in[1,\infty]$ and
 $p\in(\frac{n}{2},\infty)$ we have
\begin{align}
  \|\mathfrak{J}_2(\vec{u},\theta)\|_{\mathcal{Z}_T}\lesssim&\
    (1+T )\Big(\|\theta_0\|_{B^{-1}_{p,r}}+
  \|\vec{u}\|_{\mathcal{X}_T}\|\theta\|_{\mathcal{Z}_T}\Big).\label{eq4.3.10}
\end{align}
\end{lemma}
\begin{proof}
 Similar as before, we  divide the
proof of the $\mathfrak{J}_2(\vec{u},\theta)$ into two subcases
$q\ge0$ and $q=-1$. In the case $q\ge0$ we have
\begin{align*}
\triangle_q\mathfrak{J}_2(\vec{u},\theta)=e^{t\Delta}\triangle_q\theta_0-
\int_0^te^{(t-\tau)\Delta}{\triangle}_q\nabla\cdot(\vec{u}\theta)(\tau)d\tau.
\end{align*}
Applying Lemma \ref{lem4.2.1}, using  convolution inequalities to
time variable and following a similar argument as before we see that
\begin{align}\label{eq4.3.11}
 \|\triangle_q\mathfrak{J}_2(\vec{u},\theta)\|_{L^2_TL^{\frac{n}{2}}_x}\lesssim&
 2^{-q}\|\triangle_q\theta_0\|_{\frac{n}{2}}+\|\triangle_q(\vec{u}\theta)\|_{L^1_TL^{\frac{n}{2}}_x},
\end{align}
which  yields
\begin{align}\nonumber
\sum_{q\ge0}\|\triangle_q\mathfrak{J}_2(\vec{u},\theta)\|_{L^2_TL^{\frac{n}{2}}_x}&\lesssim\sum_{q\ge-1}(2^{-q}\|\triangle_q\theta_0\|_{\frac{n}{2}}
 +\|\triangle_q(\vec{u}\theta)\|_{L^1_TL^{\frac{n}{2}}_x})\\
 \label{eq4.3.12}&\lesssim \|\theta_0\|_{B^{-1}_{\frac{n}{2},1}}
 +\|\vec{u}\theta\|_{\tilde{L}^1_T(B^0_{\frac{n}{2},1})}
\end{align}
and
\begin{align}\label{eq4.3.13}
\sup_{q\ge0}(q+3)\|\triangle_q\mathfrak{J}_2(\vec{u},\theta)\|_{L^2_TL^{\frac{n}{2}}_x}\lesssim\|\theta_0\|_{B^{-1,1}_{\frac{n}{2},\infty}}
 +\|\vec{u}\theta\|_{\tilde{L}^1_T(B^{0,1}_{\frac{n}{2},\infty})},
\end{align}
where we have used Definition \ref{def4.1.1}. Now we consider the
case $q=-1$.  Similarly, we have
\begin{align*}
S_0\mathfrak{J}_2(\vec{u},\theta)=&e^{t\Delta}S_0\theta_0-\int_0^te^{(t-\tau)\Delta}S_0\nabla\cdot(\vec{u}\theta)(\tau)d\tau.
\end{align*}
Applying Lemma \ref{lem4.2.1} and convolution inequality to time
variable we obtain
\begin{align*}
 \|S_0\mathfrak{J}_2(\vec{u},\theta)\|_{L^2_TL^{\frac{n}{2}}_x}\lesssim
 T^{\half}\|S_0\theta_0\|_{\frac{n}{2}}
 +T^{\frac{1}{2}}\|S_0(\vec{u}\theta)\|_{L^1_TL^{\frac{n}{2}}_x}
\end{align*}
which yields
\begin{align}\label{eq4.3.14}
 \|S_0\mathfrak{J}_2(\vec{u},\theta)\|_{L^2_TL^{\frac{n}{2}}_x}&\lesssim
 T^{\frac{1}{2}}\big(\|\theta_0\|_{B^{-1}_{\frac{n}{2},\infty}}+\|\vec{u}\theta\|_{\tilde{L}^1_T(B^{0}_{\frac{n}{2},\infty})}\big)\\
 &\lesssim
 T^{\frac{1}{2}}\big(\|\theta_0\|_{B^{-1,1}_{\frac{n}{2},\infty}}+\|\vec{u}\theta\|_{\tilde{L}^1_T(B^{0,1}_{\frac{n}{2},\infty})}\big).
 \label{eq4.3.15}
\end{align}

The desired results \eqref{eq4.3.9} and \eqref{eq4.3.10} follows
from \eqref{eq4.3.12} $\sim$ \eqref{eq4.3.15} and Definition
\ref{def4.1.1} as well as Lemmas \ref{lem4.2.6} and \ref{lem4.2.7}.
This proves Lemma \ref{lem4.3.3}.
\end{proof}

\textbf{Proofs of Theorems \ref{thm4.1.2} and \ref{thm4.1.3}:} From
Lemmas \ref{lem4.3.2}, \ref{lem4.3.3} and \ref{lem4.3.1} as well as
a standard argument, we see that Theorems \ref{thm4.1.2} and
\ref{thm4.1.3} follow.

\section{Proof of Theorem \ref{thm4.1.4}}

In this section, we give the proof of Theorem \ref{thm4.1.4}. We
first prove the following heat semigroup characterization of the
space $B^{s,\sigma}_{p,r}$:

\begin{lemma}\label{lem4.4.1}
Let $p,r\in [1,\infty]$, $s<0$ and $\sigma\ge0$. The following
assertions are equivalent:

\textsc{(1)}\ \ $f\in B^{s,\sigma}_{p, r}$,.

\textsc{(2)}\ \ For all $t\in (0,1)$, $e^{t\Delta}f\in L^p_x$ and
$t^{\frac{1}{2}}|\ln(\frac{t}{e^2})|^\sigma\|e^{t\Delta}f\|_p\in
L^r((0,1),\frac{dt}{t})$.
\end{lemma}

\begin{proof}
 The idea of the proof mainly comes from \cite{Le02} and the proof is quite
 similar. But for readers convenience, we give the details as follows.
We denote by $C$ the constant depends on $n$ and might depend on
$s$, $\sigma$ and $r$ in the proof of this Lemma.

$\textsc{(1)}\Rightarrow\textsc{(2)}$.\ \  We write
$f=S_0f+\sum_{j\ge0}\triangle_jf$ with
 $$\|{S}_{0}f\|_p=\varepsilon_{-1},\quad \|\triangle_jf\|_p=2^{j|s|}(3+j)^{-\sigma}\varepsilon_j\hbox{ and } (\varepsilon_j)_{j\ge-1}\in \ell^r.$$
 We estimate the norm
 $t^{\frac{|s|}{2}}|\ln(\frac{t}{e^2})|^\sigma\|e^{t\Delta}f\|_p$ by
 \begin{align*}
  t^{\frac{|s|}{2}}|\ln(\frac{t}{e^2})|^\sigma\|e^{t\Delta}S_0f\|_p\le
  t^{\frac{|s|}{2}}|\ln(\frac{t}{e^2})|^\sigma\|e^{t\Delta}\tilde{S}_0S_0f\|_p\le
  Ct^{\frac{|s|}{2}}|\ln(\frac{t}{e^2})|^\sigma\|S_0f\|_p.
 \end{align*}
 Similarly, for $j\ge0$, we have
 \begin{align*}
  t^{\frac{|s|}{2}}|\ln(\frac{t}{e^2})|^\sigma\|e^{t\Delta}\triangle_jf\|_p\le
  Ct^{\frac{|s|}{2}}|\ln(\frac{t}{e^2})|^\sigma\|\triangle_jf\|_p.
 \end{align*}
 Moreover, when $j\ge0$ and $N\ge 0$, from the $L^1_x$ integrability of heat kernel and  Lemma \ref{lem4.2.1}, we have
 \begin{align}\nonumber
  t^{\frac{|s|}{2}}|\ln(\frac{t}{e^2})|^\sigma\|e^{t\Delta}\triangle_jf\|_p&
  =t^{\frac{|s|}{2}}|\ln(\frac{t}{e^2})|^\sigma
  \|e^{t\Delta}(-t\Delta)^{\frac{N}{2}}(-t\Delta)^{-\frac{N}{2}}\triangle_jf\|_p\\&
  \le  Ct^{\frac{|s|-N}{2}}|\ln(\frac{t}{e^2})|^\sigma2^{-jN}\|\triangle_jf\|_p\nonumber.
 \end{align}
Combining the above estimates, for some $N\ge2$ and any $0<t<1$ we
have
\begin{align*}
 \|e^{t\Delta}f\|_p&\le C\Big(
 \varepsilon_{-1}+\sum_{j\ge0}\min\Big\{{2^{j|s|}}{(j+3)^{-\sigma}}\varepsilon_j,\
 t^{-\frac{N}{2}}2^{-jN+j}(j+3)^{-\sigma}\varepsilon_j\Big
 \}\Big) \\
 &\le C\|(\varepsilon_j)_{j\ge-1}\|_{\ell^r}t^{-\frac{N}{2}}<\infty.
 \end{align*}
  Let $I_{k}=(2^{-2-2k}, 2^{-2k}]$.
 For any $t\in(0,1]=\cup_{k\ge0} I_k$, there exists an integer $j_0$ such that
 $t\in I_{j_0}$. And for  $t\in I_{j_0}$ we have
 \begin{align*}
  t^{\frac{|s|}{2}}|\ln(\frac{t}{e^2})|^\sigma\|e^{t\Delta}f\|_p\le& C \Big(
  \varepsilon_{-1}+\sum_{j=0}^{j_0}\frac{(3+j_0)^\sigma}{(3+j)^\sigma}2^{-(j_0-j)|s|}\varepsilon_j
  +\sum_{j=j_0+1}\!\!2^{(j_0-j)(|s|-N)}\varepsilon_j\Big)
  :=C\ \!\eta_{j_0}.
 \end{align*}
  From the above estimate we see that
  $$\|\|t^{\frac{|s|}{2}}|\ln(\frac{t}{e^2})|^\sigma e^{t\Delta}f\|_p\|_{L^r((0,1),\frac{dt}{t})}\le C\ \! \|(\eta_j)_{j\ge-1}\|_{\ell^r}\le C\
   \! \|(\varepsilon_j)_{j\ge-1}\|_{\ell^r}.$$
 Indeed, by using Young's inequality  we have
 \begin{align*}
  \sum_{j_0\ge-1}\Big(\sum_{j=0}^{j_0}\frac{(3+j_0)^\sigma}{(3+j)^\sigma}2^{(j-j_0)|s|}\varepsilon_j\Big)^r &\le
  C\!\!\sum_{j_0\ge-1}\Big(\sum_{j=0}^{[\frac{j_0}{2}]}2^{(j-[\frac{j_0}{2}])|s|}\varepsilon_j\Big)^r +
  C\sum_{j_0\ge-1}\Big(\sum_{j=[\frac{j_0}{2}]+1}^{j_0}\!\!2^{(j-j_0)|s|}\varepsilon_j\Big)^r
 \\&  \le C\ \! \|(\varepsilon_j)_{j\ge-1}\|_{\ell^r}^r\end{align*}
  and
  \begin{align*}
  \sum_{j_0\ge-1}\Big(\sum_{j=j_0+1}2^{(j_0-j)(|s|-N)}\varepsilon_j\Big)^r\le C \ \! \|(\varepsilon_j)_{j\ge-1}\|_{\ell^r}^r.
 \end{align*}

 $\textsc{(2)}\Rightarrow \textsc{(1)}$.\  \ We get that $S_0f=e^{-\frac{1}{2}\Delta}S_0e^{\frac{1}{2}\Delta}f\in
 L^p_x$ since the kernel of $e^{-\frac{1}{2}\Delta}S_0$ is $L^1_x$
 bounded. Similarly, when $j\ge0$, we write
 $\triangle_jf=e^{-t\Delta}\triangle_je^{t\Delta}f$. For any
 $j\ge0$, we choose $t$ such that $2^{-2-2j}<t<2^{-2j}$. Then  we have
 \begin{align*}
 2^{js}(3+j)^{\sigma}\|\triangle_jf\|_p\le Ct^{\frac{|s|}{2}}(2-\ln
 t)^\sigma\|e^{-t\Delta}\triangle_j e^{t\Delta}f\|_p\le C
 t^{\frac{|s|}{2}}|\ln(\frac{t}{e^2})|^\sigma\|e^{t\Delta}f\|_p.
 \end{align*}
 Consequently, we have
 \begin{align*}
   \|f\|_{B^{s,\sigma}_{p,r}}^r&=\sum_{j\ge-1}2^{jsr}(3+j)^{\sigma{r}}\|\triangle_jf\|^r_{p}
  \le C\Big(\sup_{0<t<1}t^{\frac{|s|}{2}}|\ln(\frac{t}{e^2})|^\sigma\|e^{t\Delta}f\|_p\Big)^r\\&
  \le
  C\|t^{\frac{|s|}{2}}|\ln(\frac{t}{e^2})|^\sigma\|e^{t\Delta}f\|_p\|^r_{L^r((0,1),\frac{dt}{t})},
 \end{align*}
 where the last inequality follows from a similar argument as in
 Lemma 16.1 of \cite{Le02}.
\end{proof}

Next we prove a bilinear estimate.

\begin{lemma}\label{lem4.4.2}
 Let $n\ge2$, $\varepsilon>0$, $p\in(\frac{n}{2},\infty)$ and $\mathfrak{J}_1$,$\mathfrak{J}_2$ be as in \eqref{eq4.3.0}. For $0<T\le1$,
 there
 exists $0<\nu<\nu(p)=\frac{2p-n}{2p}$ such that
\begin{align*}\left\{\begin{aligned}
 \sup_{t\in(0,T)}\!\!t^{\frac{1}{2}}|\ln(\frac{t}{e^2})|\|\mathfrak{J}_1(\vec{u},\theta)\|_\infty&\lesssim
 \|\vec{u}_0\|_{B^{-1,1}_{\infty,\infty}}
 +(\sup_{t\in(0,T)}\!\!\!t^{\frac{1}{2}}|\ln(\frac{t}{e^2})|\|\vec{u}\|_\infty)^2+ T^\nu\!\!\!\!\sup_{t\in(0,T)}\!\!\!
 t^{\frac{1}{2}}|\ln(\frac{t}{e^2})|^{\varepsilon}
 \|\theta\|_p\\
\sup_{t\in(0,T)}\!\!t^{\frac{1}{2}}|\ln(\frac{t}{e^2})|^\varepsilon\|\mathfrak{J}_2(\vec{u},\theta)\|_p&\lesssim \|\theta_0\|_{B^{-1,\varepsilon}_{p,\infty}}
 +\frac{1}{\varepsilon}\sup_{t\in(0,T)}t^{\frac{1}{2}}|\ln(\frac{t}{e^2})|\|\vec{u}\|_\infty\sup_{t\in(0,T)}t^{\frac{1}{2}}|\ln(\frac{t}{e^2})|^{\varepsilon}
 \|\theta\|_p \ \ .
\end{aligned}\right.\end{align*}
\end{lemma}

\begin{proof}
The term
$\int_0^te^{(t-\tau)\Delta}\mathbb{P}\nabla\cdot(\vec{u}\otimes\vec{u})d\tau$
is already treated in Lemma 2.5 of \cite{Cui11}. From Lemma
\ref{lem4.4.1} we see that
\begin{align*} \sup_{t\in(0,1)}\!\!t^{\frac{1}{2}}|\ln(\frac{t}{e^2})|\|e^{t\Delta}\vec{u}_0\|_\infty\sim \|\vec{u}_0\|_{B^{-1,1}_{\infty,\infty}},
\quad
\sup_{t\in(0,1)}\!\!t^{\frac{1}{2}}|\ln(\frac{t}{e^2})|^\varepsilon\|e^{t\Delta}\theta_0\|_p\sim
\|\theta_0\|_{B^{-1,\varepsilon}_{p,\infty}}.\end{align*} Therefore,
it remains to estimate
\begin{align*}
\int_0^te^{(t-\tau)\Delta}\mathbb{P}(\theta\vec{a})d\tau, \quad
 \int_0^te^{(t-\tau)\Delta}\nabla\cdot(\vec{u}\theta)d\tau.\end{align*}
 By using the decay estimates of the Oseen kernel (cf. \cite{Le02}, Proposition 11.1) we see that
 \begin{align*}
 t^{\frac{1}{2}}|\ln(\frac{t}{e^2})| \|\int_0^t&e^{(t-\tau)\Delta}\mathbb{P}(\theta\vec{a})d\tau\|_\infty\lesssim
  t^{\frac{1}{2}}|\ln(\frac{t}{e^2})|\int^t_0(t-\tau)^{-\frac{n}{2p}}\|\theta(\tau)\|_{p}d\tau
 \\ &\lesssim
  t^{\frac{1}{2}}|\ln(\frac{t}{e^2})|\int^t_0(t-\tau)^{-\frac{n}{2p}}\tau^{-\frac{1}{2}}|\ln(\frac{\tau}{e^2})|^{-\varepsilon}d\tau
   \sup_{\tau\in(0,T)}\!
 \tau^{\frac{1}{2}}|\ln(\frac{\tau}{e^2})|^{\varepsilon}\|\theta(\tau)\|_p
 \\ &\lesssim  t^{\nu(p)}|\ln(\frac{t}{e^2})|^{1-\varepsilon}
   \sup_{\tau\in(0,T)}\!
 \tau^{\frac{1}{2}}|\ln(\frac{\tau}{e^2})|^{\varepsilon}\|\theta(\tau)\|_p\\&\lesssim
  T^\nu\tau^{\frac{1}{2}}|\ln(\frac{\tau}{e^2})|^{\varepsilon}\|\theta(\tau)\|_p,
 \end{align*}
where $\nu\in(0,\nu({p}))$ and
$\lim_{t\downarrow0}t^{\nu(p)-\nu}|\ln(\frac{t}{e^2})|^{1-\varepsilon}=0$.
 By applying the decay estimates of the heat kernel (is Oseen kernel too) (cf. \cite{Le02}, Proposition 11.1) we see that
 \begin{align*}
 &t^{\frac{1}{2}}|\ln(\frac{t}{e^2})|^\varepsilon\|\int_0^te^{(t-\tau)\Delta}\nabla\cdot(\theta\vec{u})d\tau\|_p\lesssim
  t^{\frac{1}{2}}|\ln(\frac{t}{e^2})|^\varepsilon\int^t_0(t-\tau)^{-\frac{1}{2}}\|\theta(\tau)\|_{p}\|\vec{u}(\tau)\|_{\infty}d\tau
 \\ \!&\lesssim\!\frac{\sqrt t}{\varepsilon}
  |\ln(\frac{t}{e^2})|^\varepsilon\!\!
 \int^t_0\!(t-\tau)^{-\frac{1}{2}}\tau^{-1}|\ln(\frac{\tau}{e^2})|^{-1-\varepsilon}d\tau\!\!
   \sup_{\tau\in(0,T)}\!\! \tau^{\frac{1}{2}}|\ln(\frac{\tau}{e^2})|^{\varepsilon}\|\theta(\tau)\|_p \sup_{\tau\in(0,T)}\!\!\!
 \tau^{\frac{1}{2}}|\ln(\frac{\tau}{e^2})|\|\vec{u}(\tau)\|_\infty
 \\ &\lesssim  \frac{1}{\varepsilon}  \sup_{\tau\in(0,T)}\!
 \tau^{\frac{1}{2}}|\ln(\frac{\tau}{e^2})|^{\varepsilon}\|\theta(\tau)\|_p \sup_{\tau\in(0,T)}\!
 \tau^{\frac{1}{2}}|\ln(\frac{\tau}{e^2})|\|\vec{u}(\tau)\|_\infty,
 \end{align*}
where $t^{\frac{1}{2}}|\ln(\frac{t}{e^2})|^\varepsilon
 \int^t_0(t-\tau)^{-\frac{1}{2}}\tau^{-1}|\ln(\frac{\tau}{e^2})|^{-1-\varepsilon}d\tau\lesssim \frac{1}{\varepsilon}$. Indeed, it suffices to show
 $$t^{\frac{1}{2}}|\ln(\frac{t}{e^2})|^\varepsilon
 \int^{t/2}_0(t-\tau)^{-\frac{1}{2}}\tau^{-1}|\ln(\frac{\tau}{e^2})|^{-1-\varepsilon}d\tau\lesssim \frac{1}{\varepsilon},$$
 which is equivalent to
 $$|\ln(\frac{t}{e^2})|^\varepsilon
 \int^{t/2}_0\tau^{-1}|\ln(\frac{\tau}{e^2})|^{-1-\varepsilon}d\tau\lesssim \frac{1}{\varepsilon}
 \Leftrightarrow\int_{|\ln(\frac{t}{e^2})|}^\infty
 x^{-1-\varepsilon}dx\lesssim \frac{1}{\varepsilon}|\ln(\frac{t}{e^2})|^{-\varepsilon}.$$
 \end{proof}

\textbf{Proof of Theorem \ref{thm4.1.4}:} Applying Lemmas
\ref{lem4.3.1} and \ref{lem4.4.2}  and following similar arguments
as in \cite{Cui11}, we prove Theorem \ref{thm4.1.4}. We omit the
details here.


\bibliographystyle{amsplain}
\end{document}